\documentclass{article}

\usepackage{amsmath,amssymb,amsthm}
\usepackage{geometry}
\usepackage{hyperref}
\usepackage[T1]{fontenc}
\usepackage{indentfirst}

\newtheorem{theorem}{Theorem}
\newtheorem{lemma}{Lemma}
\newtheorem{corollary}{Corollary}
\newtheorem{conjecture}{Conjecture}

\begin{document}

\title{Normal Behavior and Periodic Points of a Pseudo-Aliquot Map Associated with the Dedekind Psi Function}

\author{Aimin Guo}

\date{}

\maketitle
\vspace{-2.0em}

\begin{center}
\small

\textit{School of Mathematics and Statistics}\\
\textit{Anhui Normal University}\\
\textit{Wuhu 241002, Anhui, P.~R.~China}

\vspace{0.4em}

\texttt{18339681864@163.com}

\end{center}

\begin{abstract}
Let $\psi$ denote the Dedekind psi function, and define
\[
T(n)=\psi(n)-n.
\]
We prove that
\[
\frac{T^{(2)}(n)}{T(n)}
=
\frac{T(n)}{n}+o(1)
\]
for almost all positive integers $n$. More generally, for every fixed
positive integer $k$, we show that
\[
\frac{T^{(j+1)}(n)}{T^{(j)}(n)}
\ge
(1-o(1))\frac{T(n)}{n},
\qquad 0\le j<k,
\]
simultaneously for almost all $n$.

As an application, for every fixed positive integer $\ell$, we prove
that
\[
\#\{n\le x:T^{(\ell)}(n)=n\}=o(x).
\]
In particular, the set of periodic points of $T$ of any prescribed
fixed exact period has asymptotic density zero.

The proofs rely on the persistence of small prime divisors through
finitely many iterations and on estimates for large prime divisors
introduced in the passage from $n$ to $T(n)$.
\end{abstract}



\section{Introduction}

The Dedekind psi function is defined by
\[
\psi(n)
=
n\prod_{p\mid n}\left(1+\frac1p\right).
\]
In this paper we study the iterates of the arithmetic function
\[
T(n):=\psi(n)-n.
\]
We put
\[
T^{(0)}(n)=n
\]
and define recursively
\[
T^{(j+1)}(n)
=
T\bigl(T^{(j)}(n)\bigr)
\]
whenever $T^{(j)}(n)$ is a positive integer. Since $T(1)=0$, an
orbit may in principle terminate at $0$. This causes no difficulty
for the almost-all results considered below: for every prescribed
fixed number of iterations, the relevant iterates are greater than
$1$ outside a set of asymptotic density zero.

Our main objective is to understand the normal behavior of the first
finitely many iterates of $T$ for almost all positive integers $n$.

The map $T$ arises naturally in the study of pseudo-aliquot
sequences. In 1993, Penney and Pomerance introduced
\[
t(n)
=
\sum_{\substack{d\mid n,\ d<n\\ n/d\ {\rm squarefree}}}d
\]
and asked whether there exists an integer whose forward orbit
\[
n,\ t(n),\ t^{(2)}(n),\ldots
\]
is unbounded \cite{PenneyPomerance1993}. In 1996, Brown and
Vanden Eynden answered this question affirmatively
\cite{BrownVandenEynden1996}. They observed that
\[
t(n)
=
n\left(
\prod_{p\mid n}\left(1+\frac1p\right)-1
\right)
=
\psi(n)-n.
\]
Thus the pseudo-aliquot map considered in these works is precisely
the function $T$ studied here.

Brown and Vanden Eynden thus settled the existence question posed by
Penney and Pomerance. What remains largely unexplored, however, is
the typical behavior of the pseudo-aliquot map over a fixed number of
iterations. In particular, we are not aware of a corresponding
almost-all analysis of the successive ratios
\[
\frac{T^{(j+1)}(n)}{T^{(j)}(n)}.
\]
The present paper addresses this gap by developing quantitative
estimates for these ratios over any prescribed fixed number of
iterations.

There is a natural comparison with the classical aliquot map
\[
s(n)=\sigma(n)-n,
\]
where $\sigma$ denotes the sum-of-divisors function. In 1976,
Erd\H{o}s investigated the almost-all behavior of the first finitely
many iterates of $s$ \cite{Erdos1976}. In particular, his work
established substantial fixed-iterate control in terms of the initial
ratio $s(n)/n$. In 2009, Kobayashi, Pollack, and Pomerance studied
the statistical distribution of sociable numbers and periodic aliquot
orbits, making explicit use of fixed-iterate estimates in this setting
\cite{KobayashiPollackPomerance2009}.

A particularly close analogue is provided by the cototient map
\[
u(n):=n-\varphi(n),
\]
where $\varphi$ denotes Euler's totient function. Luca and Pomerance
obtained a second-iterate stability result for $u$ analogous to our
first theorem \cite{ref8}. Nevertheless, the pseudo-aliquot map $T$
is not decreasing, and its fixed iterates require a separate analysis
of how prime divisors propagate under iteration. This also leads to
periodic-point consequences that have no analogue in the strictly
decreasing cototient setting.

Our first result shows that the relative growth factor is
asymptotically preserved at the second step. More precisely, we prove
that
\[
\frac{T^{(2)}(n)}{T(n)}
=
\frac{T(n)}{n}+o(1)
\]
for almost all positive integers $n$.

For $r\ge1$, we write
\[
\log_1x=\log x,
\qquad
\log_{r+1}x=\log(\log_r x).
\]
More quantitatively, on a set of asymptotic density one,
\[
\left|
\frac{T^{(2)}(n)}{T(n)}
-
\frac{T(n)}{n}
\right|
\ll
\frac{(\log_3n)^2}{\log_2n}.
\]
We do not attempt to optimize this error term. The main point is that
the relative growth factor governing the first step is typically
preserved, up to a quantity tending to zero, at the second step.

We next obtain a finite-iterate extension. For every fixed positive
integer $k$, we prove that, for almost all $n$ and simultaneously for
all $0\le j<k$,
\[
\frac{T^{(j+1)}(n)}{T^{(j)}(n)}
\ge
(1-o(1))\frac{T(n)}{n},
\]
where the $o(1)$ term may depend on $k$ and is uniform in
$0\le j<k$. Consequently, for every fixed positive integer $j$,
\[
T^{(j)}(n)
\ge
(1-o(1))\,n
\left(\frac{T(n)}{n}\right)^j
\]
for almost all positive integers $n$. Thus, over any prescribed fixed
number of iterations, the successive relative growth factors cannot
typically fall substantially below the initial one.

The proofs rely on two complementary arithmetic mechanisms. The first
is the persistence of small prime divisors under iteration. Using an
Erd\H{o}s-type prime-chain argument, we show that, for any prescribed
finite set of primes and any fixed number of iterations, the
corresponding divisibility information is typically preserved. In
particular, for a suitable fixed modulus $M$ one obtains congruences
of the form
\[
T^{(j)}(n)
\equiv
(-1)^jn
\pmod M
\]
through a prescribed fixed number of iterates.

This persistence interacts directly with the Euler product
\[
\frac{\psi(m)}m
=
\prod_{p\mid m}\left(1+\frac1p\right).
\]
Hence the small-prime factors occurring in the normalized values
$\psi(T^{(j)}(n))/T^{(j)}(n)$ are typically the same as those
occurring in $\psi(n)/n$. Since every factor $1+1/p$ is greater
than $1$, this structure is particularly well suited to propagating
lower bounds through finitely many iterates.

The second mechanism controls large prime divisors introduced in the
passage from $n$ to $T(n)$. If
\[
n=Pm,
\]
where $P$ is a prime with $P\nmid m$, then
\[
T(Pm)
=
P\,T(m)+\psi(m).
\]
Thus, under the relevant coprimality conditions, a divisibility
condition on $T(n)$ can be converted into a residue-class condition
on the large prime factor $P$. This allows the contribution of newly
introduced large prime divisors to be estimated by results on primes
in arithmetic progressions, in particular the Brun--Titchmarsh
inequality. Together, the persistence of small prime divisors and the
control of newly introduced large primes form the basis of our
normal-order estimates.

We finally give an application to periodic points. A positive integer
$n$ is called a periodic point of $T$ if
\[
T^{(r)}(n)=n
\]
for some positive integer $r$ for which the relevant iterates are
defined. The least such $r$ is called the exact period of $n$. If
$r$ is the exact period, the finite set
\[
\{n,T(n),\ldots,T^{(r-1)}(n)\}
\]
is called the periodic orbit of $n$.

For every fixed positive integer $\ell$, we prove that
\[
\#\{n\le x:T^{(\ell)}(n)=n\}=o(x).
\]
Thus the periodic points whose exact period divides $\ell$ form a set
of asymptotic density zero; in particular, the same holds for points
of exact period $\ell$. This illustrates how the normal behavior of
finite iterates can be used to obtain nontrivial information about the
dynamics of the pseudo-aliquot map.

For $\ell=2$, the result applies in particular to the
$\psi$-amicable numbers studied by Dimitrov
\cite{Dimitrov2025}. A $\psi$-amicable pair $a,b$ satisfies
\[
T(a)=b,\qquad T(b)=a,
\]
and hence forms an exact two-cycle of $T$. Our result places this
density-zero phenomenon in the broader setting of periodic points of
arbitrary prescribed fixed period.

In the periodic-point application we also use the fact that
$\psi(n)/n$ possesses a continuous limiting distribution; see, for
example, \cite[Chapter~III.4, Theorem~1]{ref13}.

\medskip
\noindent\textbf{Notation.}
Throughout the paper, $p$, $q$, $P$, and $Q$ denote prime numbers,
and $\gamma$ denotes the Euler--Mascheroni constant. For a finite set
$A$, the notation $\#A$ denotes its cardinality. We write $\pi(x)$
for the number of primes not exceeding $x$.

The symbols $O$ and $\ll$ have their usual meanings; thus
$f\ll g$ means $f=O(g)$. We write
\[
f\asymp g
\]
if both $f\ll g$ and $g\ll f$ hold. Unless otherwise specified,
all implied constants are absolute. We write
\[
p^a\parallel n
\]
if
\[
p^a\mid n
\qquad\text{and}\qquad
p^{a+1}\nmid n.
\]

All expressions involving iterated logarithms are understood for
sufficiently large arguments. Their values at the finitely many
remaining positive integers may be defined arbitrarily, since this
does not affect any asymptotic-density statement.

For a set $A$ of positive integers, write
\[
\underline d(A)
:=
\liminf_{x\to\infty}
\frac{1}{x}\#\{n\le x:n\in A\},
\]
and
\[
\overline d(A)
:=
\limsup_{x\to\infty}
\frac{1}{x}\#\{n\le x:n\in A\}.
\]
If these two quantities are equal, their common value is denoted by
$d(A)$ and is called the asymptotic density of $A$.
The phrase ``for almost all positive integers'' means outside a set
of asymptotic density zero.

Whenever a statement involves a successive ratio
\[
\frac{T^{(j+1)}(n)}{T^{(j)}(n)},
\]
integers for which one of the relevant iterates is undefined or the
denominator is zero are understood to belong to the exceptional set.

\section{Main Results}

We first state the asymptotic result for the second iterate.

\begin{theorem}\label{thm:second}
For almost all positive integers $n$,
\[
\frac{T^{(2)}(n)}{T(n)}
=
\frac{T(n)}{n}+o(1).
\]
More precisely, on a set of positive integers of asymptotic density
$1$,
\[
\left|
\frac{T^{(2)}(n)}{T(n)}
-
\frac{T(n)}{n}
\right|
\ll
\frac{(\log_3 n)^2}{\log_2 n}.
\]
\end{theorem}

Our second result gives a lower bound for any fixed number of
iterations.

\begin{theorem}\label{thm:lower}
Let $k$ be a fixed positive integer. Then, for almost all positive
integers $n$, simultaneously for all integers $0\le j<k$,
\[
\frac{T^{(j+1)}(n)}{T^{(j)}(n)}
\ge
(1-o(1))\frac{T(n)}{n}.
\]
Here the $o(1)$ term may depend on $k$ and is uniform in
$0\le j<k$.

Consequently, for every fixed positive integer $j$,
\[
T^{(j)}(n)
\ge
(1-o(1))\,n
\left(\frac{T(n)}{n}\right)^j
\]
for almost all positive integers $n$.
\end{theorem}

\begin{corollary}\label{cor:periodic}
Let $\ell$ be a fixed positive integer. Then
\[
\#\{n\le x:T^{(\ell)}(n)=n\}=o(x).
\]
In particular, for every fixed $\ell$, the set of periodic points
of $T$ of exact period $\ell$ has asymptotic density zero.
\end{corollary}

\section{Preliminary Results}
\label{sec:preliminaries}

Throughout this section, for all sufficiently large $n$, let
\[
g(n)
=
c\,\frac{\log_2 n}{\log_3 n},
\qquad
\delta(n)
=
\frac{\log_3 n}{\log_2 n},
\]
where $c>0$ is a suitable fixed absolute constant.

We first estimate the contribution of large prime divisors and then
establish the persistence of small prime divisors under any fixed
number of iterations.

\begin{lemma}\label{lem:old-large-primes}
Define
\[
h_1(n)
:=
\sum_{\substack{p>g(n)\\p\mid n}}\frac1p.
\]
Then
\[
h_1(n)<\delta(n)
\]
for a set of positive integers $n$ of asymptotic density $1$.
\end{lemma}

\begin{proof}
Let $x$ be sufficiently large. Since
\[
t\longmapsto\frac{\log_2t}{\log_3t}
\]
is increasing for all sufficiently large $t$, for
\[
\sqrt{x}<n\le x
\]
we have
\[
g(n)\ge g(\sqrt{x}).
\]
Hence
\[
h_1(n)
\le
\sum_{\substack{p>g(\sqrt{x})\\p\mid n}}\frac1p.
\]
Therefore
\[
\begin{aligned}
\sum_{\sqrt{x}<n\le x}h_1(n)
&\le
\sum_{\sqrt{x}<n\le x}
\sum_{\substack{p>g(\sqrt{x})\\p\mid n}}\frac1p\\
&\le
\sum_{p>g(\sqrt{x})}\frac1p
\#\{n\le x:p\mid n\}\\
&\le
x\sum_{p>g(\sqrt{x})}\frac1{p^2}.
\end{aligned}
\]
By partial summation together with the standard estimate
\[
\pi(t)\ll\frac{t}{\log t},
\]
we have, for sufficiently large $y$,
\[
\sum_{p>y}\frac1{p^2}
\ll
\frac1{y\log y}.
\]
Since
\[
g(\sqrt{x})
=
c\,\frac{\log_2\sqrt{x}}{\log_3\sqrt{x}}
\asymp
\frac{\log_2x}{\log_3x}
\]
and
\[
\log g(\sqrt{x})\asymp\log_3x,
\]
it follows that
\[
\sum_{\sqrt{x}<n\le x}h_1(n)
\ll
\frac{x}{\log_2x}.
\]

Now let
\[
\mathcal E(x)
:=
\left\{
n:\sqrt{x}<n\le x,\quad
h_1(n)\ge\delta(n)
\right\}.
\]
Uniformly for $\sqrt{x}<n\le x$,
\[
\delta(n)
=
\frac{\log_3n}{\log_2n}
\asymp
\frac{\log_3x}{\log_2x}.
\]
Thus there exists an absolute constant $c_0>0$ such that, for all
sufficiently large $x$ and every $n\in\mathcal E(x)$,
\[
h_1(n)
\ge
c_0\frac{\log_3x}{\log_2x}.
\]
Consequently,
\[
c_0\,\#\mathcal E(x)
\frac{\log_3x}{\log_2x}
\le
\sum_{\sqrt{x}<n\le x}h_1(n)
\ll
\frac{x}{\log_2x},
\]
and hence
\[
\#\mathcal E(x)
\ll
\frac{x}{\log_3x}
=o(x).
\]
Therefore
\[
\#\{n\le x:h_1(n)\ge\delta(n)\}
\le
\sqrt{x}+\#\mathcal E(x)
=o(x).
\]
Hence
\[
h_1(n)<\delta(n)
\]
for a set of positive integers of asymptotic density $1$.
\end{proof}

We next control the large prime divisors newly introduced by $T$.

\begin{lemma}\label{lem:new-large-primes}
Define
\[
\mathcal H(n)
:=
\sum_{\substack{
p>g(n)\\
p\mid T(n),\;p\nmid n}}
\frac1p.
\]
Then
\[
\mathcal H(n)<\delta(n)
\]
for a set of positive integers $n$ of asymptotic density $1$.
\end{lemma}

\begin{proof}
Let $x$ be sufficiently large, and put
\[
\eta=\eta(x)
:=
\frac{\log_4x}{3\log_3x}.
\]
For every integer $n>1$, let $P(n)$ denote the largest prime factor
of $n$. Define
\[
\mathcal A(x)
:=
\left\{
n:
\sqrt{x}<n\le x,\quad
P(n)>x^\eta,\quad
P(n)^2\nmid n
\right\}.
\]

For $\sqrt{x}<n\le x$, we have
\[
g(n)\le g(x)\ll\log_2x,
\]
whereas
\[
x^{\eta/2}
=
\exp\left(
\frac{\log x\,\log_4x}{6\log_3x}
\right).
\]
Hence, for all sufficiently large $x$,
\[
g(n)<x^{\eta/2}
\qquad
(\sqrt{x}<n\le x).
\]
Thus, for $n\in\mathcal A(x)$, the primes occurring in
$\mathcal H(n)$ split into the two ranges
\[
g(n)<p\le x^{\eta/2}
\qquad\text{and}\qquad
p>x^{\eta/2}.
\]

Accordingly, we write
\begin{equation}
\sum_{\sqrt{x}<n\le x}\mathcal H(n)
=
S_1+S_2+S_3,
\label{eq:H-decomposition}
\end{equation}
where
\[
S_1
:=
\sum_{\substack{
\sqrt{x}<n\le x\\
n\notin\mathcal A(x)}}
\mathcal H(n),
\]
\[
S_2
:=
\sum_{n\in\mathcal A(x)}
\sum_{\substack{
p>x^{\eta/2}\\
p\mid T(n),\;p\nmid n}}
\frac1p,
\]
and
\[
S_3
:=
\sum_{n\in\mathcal A(x)}
\sum_{\substack{
g(n)<p\le x^{\eta/2}\\
p\mid T(n),\;p\nmid n}}
\frac1p.
\]
We estimate these three sums separately.

\medskip
\noindent
\textit{Estimate of $S_1$.}
The complement of $\mathcal A(x)$ in the interval
$\sqrt{x}<n\le x$ is the disjoint union
\[
\mathcal B_2(x)\cup\mathcal B_3(x),
\]
where
\[
\mathcal B_2(x)
:=
\left\{
n:
\sqrt{x}<n\le x,\quad
P(n)\le x^\eta
\right\},
\]
and
\[
\mathcal B_3(x)
:=
\left\{
n:
\sqrt{x}<n\le x,\quad
P(n)>x^\eta,\quad
P(n)^2\mid n
\right\}.
\]

For $\mathcal B_2(x)$, de Bruijn's estimate for smooth numbers
\cite{deBruijn} gives
\[
\#\mathcal B_2(x)
\le
x\exp\left\{
-(1+o(1))
\frac1\eta\log\frac1\eta
\right\}.
\]
Since
\[
\frac1\eta
=
\frac{3\log_3x}{\log_4x}
\]
and
\[
\log\frac1\eta
=
\log\left(
\frac{3\log_3x}{\log_4x}
\right)
=
\log_4x-\log_5x+O(1),
\]
we have
\[
\frac1\eta\log\frac1\eta
=
(3+o(1))\log_3x.
\]
Consequently,
\[
\#\mathcal B_2(x)
\ll
\frac{x}{(\log_2x)^{5/2}}
\]
for all sufficiently large $x$.

Moreover,
\[
\begin{aligned}
\#\mathcal B_3(x)
&\le
\sum_{p>x^\eta}
\#\{n\le x:p^2\mid n\}\\
&\le
x\sum_{p>x^\eta}\frac1{p^2}\\
&\le
x\sum_{m>x^\eta}\frac1{m^2}\\
&\ll
x^{1-\eta}.
\end{aligned}
\]

We next obtain a uniform upper bound for $\mathcal H(n)$ in the
range $\sqrt{x}<n\le x$. Since
\[
\psi(n)\le\sigma(n)
\]
and
\[
\frac{\sigma(n)}n
=
\sum_{d\mid n}\frac1d
\le
1+\log n,
\]
we have
\[
0<T(n)<\psi(n)\ll x\log x.
\]
Hence every prime divisor of $T(n)$ is at most $Cx\log x$ for some
absolute constant $C>0$. By Mertens' estimate \cite{ref9},
\[
\mathcal H(n)
\le
\sum_{p\mid T(n)}\frac1p
\le
\sum_{p\le Cx\log x}\frac1p
\ll
\log_2x
\]
uniformly for $\sqrt{x}<n\le x$.

It follows that
\[
S_1
\ll
\left(
\frac{x}{(\log_2x)^{5/2}}
+
x^{1-\eta}
\right)\log_2x.
\]
Since
\[
x^{-\eta}(\log_2x)^2\longrightarrow0,
\]
we obtain
\begin{equation}
S_1
\ll
\frac{x}{\log_2x}.
\label{eq:H-S1}
\end{equation}

\medskip
\noindent
\textit{Estimate of $S_2$.}
Fix $n\in\mathcal A(x)$, and suppose that $T(n)$ has $r$ distinct
prime divisors
\[
p_1,\ldots,p_r>x^{\eta/2}.
\]
Since
\[
p_1\cdots p_r\mid T(n),
\]
we have
\[
x^{r\eta/2}
<
p_1\cdots p_r
\le
T(n)
\ll
x\log x.
\]
Taking logarithms gives
\[
\frac{r\eta}{2}\log x
\le
\log x+\log_2x+O(1),
\]
and hence, for all sufficiently large $x$,
\[
r\le\frac3\eta.
\]
Therefore
\[
\sum_{\substack{
p>x^{\eta/2}\\
p\mid T(n),\;p\nmid n}}
\frac1p
\le
\frac3\eta x^{-\eta/2}.
\]
Summing over $n\in\mathcal A(x)$, we obtain
\[
S_2
\ll
\frac{x^{1-\eta/2}}{\eta}.
\]
Since
\[
x^{-\eta/2}
=
\exp\left(
-\frac{\log x\,\log_4x}
{6\log_3x}
\right)
\]
decays faster than every fixed negative power of $\log x$, we get
\begin{equation}
S_2
\ll
\frac{x}{\log x}.
\label{eq:H-S2}
\end{equation}

\medskip
\noindent
\textit{Estimate of $S_3$.}
Since
\[
n>\sqrt{x}
\qquad
(n\in\mathcal A(x)),
\]
and $g$ is increasing for all sufficiently large arguments, we have
\[
g(n)\ge g(x^{1/2}).
\]
Hence
\[
S_3
\le
\sum_{g(x^{1/2})<p\le x^{\eta/2}}
\frac1p\,N_p(x),
\]
where
\[
N_p(x)
:=
\#\left\{
n\in\mathcal A(x):
p\mid T(n),\ p\nmid n
\right\}.
\]

Fix a prime $p$ satisfying
\[
g(x^{1/2})<p\le x^{\eta/2}.
\]
Let $n$ be counted by $N_p(x)$, and write
\[
n=Pm,
\qquad
P=P(n).
\]
Since
\[
P^2\nmid n,
\]
we have
\[
P\nmid m.
\]
Moreover, since
\[
P>x^\eta
\qquad\text{and}\qquad
n\le x,
\]
we have
\[
m<x^{1-\eta}.
\]

By the multiplicativity of $\psi$ and $P\nmid m$,
\begin{equation}
T(Pm)
=
P\,T(m)+\psi(m).
\label{eq:H-linear-identity}
\end{equation}
Since $p\mid T(n)$, we obtain
\begin{equation}
P\,T(m)
\equiv
-\psi(m)
\pmod p.
\label{eq:H-congruence}
\end{equation}

We claim that
\[
p\nmid T(m)\psi(m).
\]
Indeed, suppose first that
\[
p\mid T(m).
\]
Then \eqref{eq:H-congruence} gives
\[
p\mid\psi(m).
\]
Since
\[
T(m)=\psi(m)-m,
\]
it follows that
\[
p\mid m.
\]
But $m\mid n$, contradicting
\[
p\nmid n.
\]

Suppose next that
\[
p\mid\psi(m).
\]
By the preceding argument,
\[
p\nmid T(m).
\]
Hence \eqref{eq:H-congruence} implies
\[
p\mid P.
\]
Since both $p$ and $P$ are prime, this would give
\[
p=P.
\]
However,
\[
p
\le
x^{\eta/2}
<
x^\eta
<
P,
\]
a contradiction. Thus
\[
p\nmid T(m)\psi(m).
\]

Therefore $T(m)$ is invertible modulo $p$, and
\[
-\psi(m)T(m)^{-1}
\not\equiv
0
\pmod p.
\]
Let $a_{m,p}$ be the unique integer satisfying
\[
1\le a_{m,p}\le p-1
\]
and
\[
a_{m,p}
\equiv
-\psi(m)T(m)^{-1}
\pmod p.
\]
Then
\[
P\equiv a_{m,p}\pmod p,
\qquad
(a_{m,p},p)=1.
\]

For $(a,q)=1$, write
\[
\pi(X;q,a)
:=
\#\left\{
P\le X:
P\text{ is prime and }P\equiv a\pmod q
\right\}.
\]
Consequently,
\[
N_p(x)
\le
\sum_{\substack{
m\le x^{1-\eta}\\
p\nmid T(m)\psi(m)}}
\pi\left(
\frac{x}{m};
p,a_{m,p}
\right).
\]

Since
\[
m\le x^{1-\eta}
\qquad\text{and}\qquad
p\le x^{\eta/2},
\]
we have
\[
\frac{x}{m}
\ge
x^\eta
>
x^{\eta/2}
\ge
p.
\]
Thus the Brun--Titchmarsh inequality \cite{ref14} is applicable and
gives
\[
\pi\left(
\frac{x}{m};
p,a_{m,p}
\right)
\le
\frac{2x}
{m(p-1)\log(x/(mp))}.
\]
Moreover,
\[
mp
\le
x^{1-\eta/2},
\]
so
\[
\log\frac{x}{mp}
\ge
\frac{\eta}{2}\log x.
\]
It follows that
\[
\begin{aligned}
N_p(x)
&\le
\frac{4x}
{(p-1)\eta\log x}
\sum_{\substack{
m\le x^{1-\eta}\\
p\nmid T(m)\psi(m)}}
\frac1m\\
&\le
\frac{4x}
{(p-1)\eta\log x}
\sum_{m\le x^{1-\eta}}\frac1m\\
&\ll
\frac{x}{\eta p}.
\end{aligned}
\]

Hence
\[
S_3
\ll
\frac{x}{\eta}
\sum_{p>g(x^{1/2})}\frac1{p^2}.
\]
By partial summation together with
\[
\pi(t)\ll\frac{t}{\log t},
\]
we have
\[
\sum_{p>y}\frac1{p^2}
\ll
\frac1{y\log y}
\]
for sufficiently large $y$. Therefore
\[
S_3
\ll
\frac{x}
{\eta\,g(x^{1/2})
\log g(x^{1/2})}.
\]
Since
\[
g(x^{1/2})
\asymp
\frac{\log_2x}{\log_3x}
\]
and
\[
\log g(x^{1/2})
\asymp
\log_3x,
\]
we have
\[
g(x^{1/2})
\log g(x^{1/2})
\asymp
\log_2x.
\]
Recalling that
\[
\eta
=
\frac{\log_4x}{3\log_3x},
\]
we conclude that
\begin{equation}
S_3
\ll
\frac{x\log_3x}
{\log_2x\,\log_4x}.
\label{eq:H-S3}
\end{equation}

Combining
\eqref{eq:H-S1},
\eqref{eq:H-S2}, and
\eqref{eq:H-S3}, we obtain
\begin{equation}
\sum_{\sqrt{x}<n\le x}\mathcal H(n)
\ll
\frac{x\log_3x}
{\log_2x\,\log_4x}.
\label{eq:H-mean}
\end{equation}

Now define
\[
\mathcal E_H(x)
:=
\left\{
n:
\sqrt{x}<n\le x,\quad
\mathcal H(n)\ge\delta(n)
\right\}.
\]
Uniformly for
\[
\sqrt{x}<n\le x,
\]
we have
\[
\delta(n)
=
\frac{\log_3n}{\log_2n}
\asymp
\frac{\log_3x}{\log_2x}.
\]
Hence there exists an absolute constant $c_H>0$ such that, for all
sufficiently large $x$ and every $n\in\mathcal E_H(x)$,
\[
\mathcal H(n)
\ge
c_H
\frac{\log_3x}{\log_2x}.
\]
Therefore, by \eqref{eq:H-mean},
\[
c_H
\#\mathcal E_H(x)
\frac{\log_3x}{\log_2x}
\le
\sum_{\sqrt{x}<n\le x}\mathcal H(n)
\ll
\frac{x\log_3x}
{\log_2x\,\log_4x},
\]
and hence
\[
\#\mathcal E_H(x)
\ll
\frac{x}{\log_4x}
=
o(x).
\]

Finally, the integers $n\le\sqrt{x}$ contribute at most
$\sqrt{x}=o(x)$ additional possible exceptions. Thus
\[
\begin{aligned}
\#\{n\le x:\mathcal H(n)\ge\delta(n)\}
&\le
\sqrt{x}
+
\#\mathcal E_H(x)\\
&=
o(x).
\end{aligned}
\]
Therefore
\[
\mathcal H(n)<\delta(n)
\]
for a set of positive integers of asymptotic density $1$.
\end{proof}

We now turn to the persistence of small prime divisors under a fixed
number of iterations. The parameters below are always kept fixed
before the density limit is taken.

\begin{lemma}\label{lem:prime-chain}
Let $M$ and $\ell$ be fixed positive integers. Then, for almost all
positive integers $n$, there exist prime divisors
\[
q_1,\ldots,q_\ell
\]
of $n$ such that
\[
q_1\equiv-1\pmod M
\]
and
\[
q_{r+1}\equiv-1\pmod{q_r},
\qquad
1\le r<\ell.
\]
\end{lemma}

\begin{proof}
This follows directly from Lemma~1 of Erd\H{o}s~[6],
applied with the fixed parameters $t=M$ and $k=\ell$.
Indeed, that lemma gives, for almost all $n$, a chain of prime
divisors $q_1,\ldots,q_\ell$ satisfying
\[
q_1\equiv -1\pmod M,
\qquad
q_{r+1}\equiv -1\pmod{q_r},
\quad 1\le r<\ell.
\]
\end{proof}

The preceding prime chain allows the congruence information to be
propagated through finitely many iterates.

\begin{lemma}\label{lem:congruence-persistence}
Let $K$ and $M$ be fixed positive integers. Then, for almost all
positive integers $n$,
\[
T^{(j)}(n)>1
\]
and
\[
T^{(j)}(n)\equiv(-1)^jn\pmod M,
\qquad
0\le j\le K,
\]
simultaneously.
\end{lemma}

\begin{proof}
Apply Lemma~\ref{lem:prime-chain} with
\[
\ell=K+1.
\]
Outside a set of asymptotic density $0$, the integer $n$ has prime
divisors
\[
q_1,\ldots,q_{K+1}
\]
satisfying
\[
q_1\equiv-1\pmod M
\]
and
\[
q_{r+1}\equiv-1\pmod{q_r},
\qquad
1\le r\le K.
\]

We first show simultaneously that
\[
T^{(j)}(n)>1
\]
and
\[
q_1,\ldots,q_{K+1-j}\mid T^{(j)}(n),
\qquad
0\le j\le K.
\]
For $j=0$, we have
\[
T^{(0)}(n)=n.
\]
Since every $q_r$ divides $n$, and in particular $q_1\mid n$, we
have
\[
n\ge q_1\ge2.
\]
Thus the assertion holds for $j=0$.

Suppose that the assertion holds for some $j<K$. Then
\[
T^{(j)}(n)>1.
\]
Let
\[
1\le r\le K-j.
\]
By the induction hypothesis,
\[
q_r\mid T^{(j)}(n)
\qquad\text{and}\qquad
q_{r+1}\mid T^{(j)}(n).
\]
Write
\[
q_{r+1}^a\parallel T^{(j)}(n)
\]
for some $a\ge1$. Since
\[
\psi(q_{r+1}^a)
=
q_{r+1}^{a-1}(q_{r+1}+1)
\]
and
\[
q_r\mid q_{r+1}+1,
\]
the multiplicativity of $\psi$ gives
\[
q_r\mid\psi\bigl(T^{(j)}(n)\bigr).
\]
Since also
\[
q_r\mid T^{(j)}(n),
\]
we obtain
\[
q_r
\mid
\psi\bigl(T^{(j)}(n)\bigr)-T^{(j)}(n)
=
T^{(j+1)}(n).
\]
Hence
\[
q_1,\ldots,q_{K-j}\mid T^{(j+1)}(n).
\]

Moreover, $T(m)>0$ for every integer $m>1$, since
\[
T(m)
=
m\left(
\prod_{p\mid m}\left(1+\frac1p\right)-1
\right)>0.
\]
Thus
\[
T^{(j+1)}(n)>0.
\]
Since
\[
q_1\mid T^{(j+1)}(n),
\]
it follows that
\[
T^{(j+1)}(n)\ge q_1\ge2.
\]
This completes the induction.

In particular,
\[
q_1\mid T^{(j)}(n),
\qquad
0\le j\le K.
\]
Since
\[
q_1\equiv-1\pmod M,
\]
we have
\[
M\mid q_1+1.
\]
For $0\le j<K$, write
\[
q_1^a\parallel T^{(j)}(n).
\]
Then
\[
\psi(q_1^a)
=
q_1^{a-1}(q_1+1)
\]
is divisible by $M$. By the multiplicativity of $\psi$, it follows
that
\[
M\mid\psi\bigl(T^{(j)}(n)\bigr).
\]
Therefore
\[
\begin{aligned}
T^{(j+1)}(n)
&=
\psi\bigl(T^{(j)}(n)\bigr)-T^{(j)}(n)\\
&\equiv
-T^{(j)}(n)
\pmod M.
\end{aligned}
\]
Starting from
\[
T^{(0)}(n)=n,
\]
induction gives
\[
T^{(j)}(n)\equiv(-1)^jn\pmod M,
\qquad
0\le j\le K.
\]
\end{proof}

As an immediate consequence, divisibility by any prescribed finite
set of primes is preserved through a fixed number of iterations.

\begin{lemma}\label{lem:small-prime-persistence}
Let $K$ be a fixed positive integer and let $y\ge2$ be fixed.
Then, for almost all positive integers $n$,
\[
p\mid T^{(j)}(n)
\quad\Longleftrightarrow\quad
p\mid n
\]
simultaneously for every prime $p\le y$ and every
$0\le j\le K$.
\end{lemma}

\begin{proof}
Let
\[
M=\prod_{p\le y}p.
\]
Since $y$ is fixed, so is $M$. By
Lemma~\ref{lem:congruence-persistence}, for almost all $n$,
\[
T^{(j)}(n)\equiv(-1)^jn\pmod M,
\qquad
0\le j\le K.
\]
Hence, for every prime $p\le y$,
\[
T^{(j)}(n)\equiv(-1)^jn\pmod p.
\]
Since $(-1)^j$ is invertible modulo $p$, it follows that
\[
p\mid T^{(j)}(n)
\quad\Longleftrightarrow\quad
p\mid n.
\]
\end{proof}

For later use, define
\[
R(m)
:=
\frac{\psi(m)}m
=
\prod_{p\mid m}\left(1+\frac1p\right).
\]
For $y\ge2$, put
\[
R_y(m)
:=
\prod_{\substack{p\le y\\p\mid m}}
\left(1+\frac1p\right),
\qquad
H_y(m)
:=
\sum_{\substack{p>y\\p\mid m}}\frac1p.
\]

Since $R_y(m)$ depends only on the prime divisors $p\le y$ of $m$,
Lemma~\ref{lem:small-prime-persistence} implies that, for every
fixed $K$ and $y$,
\[
R_y\bigl(T^{(j)}(n)\bigr)
=
R_y(n),
\qquad
0\le j\le K,
\]
for almost all positive integers $n$.

We shall also need the following elementary estimate for the
large-prime tail.

\begin{lemma}\label{lem:large-prime-tail}
For $x\ge1$ and $y\ge2$,
\[
\sum_{n\le x}H_y(n)
\ll
\frac{x}{y}.
\]
Consequently, for every $\varepsilon>0$,
\[
\limsup_{x\to\infty}
\frac1x
\#\{n\le x:H_y(n)>\varepsilon\}
\ll
\frac1{\varepsilon y}.
\]
\end{lemma}

\begin{proof}
Interchanging the order of summation gives
\[
\begin{aligned}
\sum_{n\le x}H_y(n)
&=
\sum_{n\le x}
\sum_{\substack{p>y\\p\mid n}}\frac1p\\
&=
\sum_{p>y}\frac1p
\left\lfloor\frac{x}{p}\right\rfloor\\
&\le
x\sum_{p>y}\frac1{p^2}\\
&\le
x\sum_{m>y}\frac1{m^2}\\
&\ll
\frac{x}{y}.
\end{aligned}
\]
This proves the first assertion.

For the second assertion,
\[
\varepsilon\,
\#\{n\le x:H_y(n)>\varepsilon\}
\le
\sum_{n\le x}H_y(n).
\]
Hence
\[
\#\{n\le x:H_y(n)>\varepsilon\}
\ll
\frac{x}{\varepsilon y}.
\]
Dividing by $x$ and taking the upper limit as $x\to\infty$ gives
\[
\limsup_{x\to\infty}
\frac1x
\#\{n\le x:H_y(n)>\varepsilon\}
\ll
\frac1{\varepsilon y}.
\]
\end{proof}

\begin{lemma}\label{lem:continuous-distribution}
The arithmetic function
\[
\frac{\psi(n)}n
\]
has a continuous limiting distribution.
\end{lemma}

\begin{proof}
Set
\[
f(n)
=
\log\frac{\psi(n)}n
=
\sum_{p\mid n}\log\left(1+\frac1p\right).
\]
Then $f$ is strongly additive. Moreover,
\[
f(p)
=
\log\left(1+\frac1p\right)
\ll
\frac1p.
\]

We apply the Erd\H{o}s--Wintner theorem
\cite[Chapter~III.4, Theorem~1]{ref13},
with its truncation parameter equal to $1$.
Since
\[
0<f(p)\le\log\frac32<1
\]
for every prime $p$, we have
\[
\sum_{\substack{p\\|f(p)|>1}}\frac1p=0.
\]
Furthermore,
\[
\sum_p\frac{f(p)}p
\ll
\sum_p\frac1{p^2}
<\infty
\]
and
\[
\sum_p\frac{f(p)^2}{p}
\ll
\sum_p\frac1{p^3}
<\infty.
\]
Thus the conditions of the Erd\H{o}s--Wintner theorem are satisfied,
and hence $f$ possesses a limiting distribution.

Moreover,
\[
f(p)>0
\]
for every prime $p$. Therefore
\[
\sum_{\substack{p\\f(p)\ne0}}\frac1p
=
\sum_p\frac1p
=
\infty.
\]
By the continuity criterion in the same theorem, the limiting
distribution of $f$ is continuous.

Let $F$ denote this continuous limiting distribution function.
Since
\[
\frac{\psi(n)}n=e^{f(n)},
\]
for every $y>0$ we have
\[
\frac{\psi(n)}n\le y
\quad\Longleftrightarrow\quad
f(n)\le\log y.
\]
Hence the limiting distribution function of $\psi(n)/n$ is
\[
G(y)
=
\begin{cases}
0, & y\le0,\\[3pt]
F(\log y), & y>0.
\end{cases}
\]
Since every distribution function satisfies
\[
F(t)\longrightarrow0
\qquad(t\to-\infty),
\]
the function $G$ is continuous also at $y=0$. Therefore
$\psi(n)/n$ has a continuous limiting distribution.
\end{proof}

\section{Proof of Theorem~\ref{thm:second}}

\begin{proof}
Let $c_\ast$ be the constant occurring in
\cite[Lemma~2.3]{Guo}. We choose the constant $c$ in
\[
g(n)
=
c\,\frac{\log_2n}{\log_3n}
\]
so that
\[
0<c<c_\ast
\]
and so that the conclusions of
Lemmas~\ref{lem:old-large-primes} and
\ref{lem:new-large-primes} hold with this choice of $g$.

We first deal explicitly with the exceptional set arising from the
small-prime divisibility result.

For sufficiently large $x$, define
\[
\mathcal E_0(x)
:=
\left\{
n:\sqrt{x}<n\le x,\quad
\text{there exists a prime }p\le g(n)
\text{ such that }p\nmid\psi(n)
\right\}.
\]

We claim that
\begin{equation}
\#\mathcal E_0(x)
\ll
\frac{x}{(\log_3x)^2}.
\label{eq:psiT-small-exception}
\end{equation}
Indeed, if
\[
\sqrt{x}<n\le x,
\]
then, since $g(t)$ is increasing for all sufficiently large $t$,
\[
g(n)\le g(x)
=
c\,\frac{\log_2x}{\log_3x}
<
c_\ast\,\frac{\log_2x}{\log_3x}.
\]
Therefore every prime $p\le g(n)$ satisfies
\[
p<
c_\ast\,\frac{\log_2x}{\log_3x}.
\]
Taking $a=1$ in \cite[Lemma~2.3]{Guo}, we see that, apart from
\[
O\left(\frac{x}{(\log_3x)^2}\right)
\]
integers $n\le x$, every such prime $p$ divides $\psi(n)$.
This proves \eqref{eq:psiT-small-exception}.

Define also
\[
\mathcal E_{\mathcal H}(x)
:=
\left\{
n:\sqrt{x}<n\le x,\quad
\mathcal H(n)\ge
\frac{\log_3n}{\log_2n}
\right\},
\]
and
\[
\mathcal E_h(x)
:=
\left\{
n:\sqrt{x}<n\le x,\quad
h_1(n)\ge
\frac{\log_3n}{\log_2n}
\right\}.
\]
By Lemmas~\ref{lem:new-large-primes} and
\ref{lem:old-large-primes},
\[
\#\mathcal E_{\mathcal H}(x)=o(x),
\qquad
\#\mathcal E_h(x)=o(x).
\]

Put
\[
\mathcal E(x)
:=
\{n:n\le\sqrt{x}\}
\cup
\mathcal E_0(x)
\cup
\mathcal E_{\mathcal H}(x)
\cup
\mathcal E_h(x).
\]
Then
\begin{equation}
\#\mathcal E(x)=o(x).
\label{eq:psiT-total-exception}
\end{equation}

We now fix
\[
n\le x,
\qquad
n\notin\mathcal E(x).
\]
Then
\[
\mathcal H(n)
<
\frac{\log_3n}{\log_2n},
\qquad
h_1(n)
<
\frac{\log_3n}{\log_2n},
\]
and
\[
p\mid\psi(n)
\qquad
\text{for every prime }p\le g(n).
\]

We first compare the small prime divisors of $n$ and $T(n)$.
Since
\[
T(n)=\psi(n)-n,
\]
for every prime $p\le g(n)$ we have
\[
T(n)\equiv-n\pmod p.
\]
Consequently,
\begin{equation}
p\mid T(n)
\quad\Longleftrightarrow\quad
p\mid n
\qquad(p\le g(n)).
\label{eq:psiT-small-prime-equivalence}
\end{equation}
Thus $n$ and $T(n)$ have precisely the same prime divisors not
exceeding $g(n)$.

Recall that
\[
\frac{\psi(m)}m
=
\prod_{p\mid m}
\left(1+\frac1p\right).
\]
Therefore
\[
\frac{\psi(n)}n
=
\prod_{p\mid n}
\left(1+\frac1p\right)
\]
and
\[
\frac{\psi(T(n))}{T(n)}
=
\prod_{p\mid T(n)}
\left(1+\frac1p\right).
\]

By \eqref{eq:psiT-small-prime-equivalence}, the Euler factors
corresponding to all primes $p\le g(n)$ are identical in these two
products.

Define
\[
\mathcal P_+(n)
:=
\left\{
p>g(n):
p\mid T(n),\ p\nmid n
\right\},
\]
and
\[
\mathcal P_-(n)
:=
\left\{
p>g(n):
p\mid n,\ p\nmid T(n)
\right\}.
\]
The primes $p>g(n)$ which divide both $n$ and $T(n)$ also cancel.
Hence
\begin{equation}
\frac{\psi(T(n))/T(n)}
{\psi(n)/n}
=
\frac{
\displaystyle
\prod_{p\in\mathcal P_+(n)}
\left(1+\frac1p\right)}
{
\displaystyle
\prod_{p\in\mathcal P_-(n)}
\left(1+\frac1p\right)}.
\label{eq:psiT-euler-ratio}
\end{equation}

Taking logarithms in \eqref{eq:psiT-euler-ratio} and using
\[
\log(1+u)\le u
\qquad(u>0),
\]
we obtain
\[
\left|
\log
\frac{\psi(T(n))/T(n)}
{\psi(n)/n}
\right|
\le
\sum_{p\in\mathcal P_+(n)}\frac1p
+
\sum_{p\in\mathcal P_-(n)}\frac1p.
\]
By the definition of $\mathcal H(n)$,
\[
\sum_{p\in\mathcal P_+(n)}\frac1p
=
\mathcal H(n)
<
\frac{\log_3n}{\log_2n}.
\]
Moreover,
\[
\mathcal P_-(n)
\subseteq
\{p>g(n):p\mid n\},
\]
so that
\[
\sum_{p\in\mathcal P_-(n)}\frac1p
\le
h_1(n)
<
\frac{\log_3n}{\log_2n}.
\]
Consequently,
\begin{equation}
\left|
\log
\frac{\psi(T(n))/T(n)}
{\psi(n)/n}
\right|
<
2\frac{\log_3n}{\log_2n}.
\label{eq:psiT-log-bound}
\end{equation}

Recall that
\[
\delta(n)
=
\frac{\log_3n}{\log_2n}.
\]
Since $\delta(n)\to0$, there exists a real number $\theta_n$ such
that
\[
|\theta_n|<2\delta(n)
\]
and
\[
\frac{\psi(T(n))}{T(n)}
=
\frac{\psi(n)}n e^{\theta_n}.
\]

We now convert this multiplicative estimate into an additive one.
For this purpose, we first obtain an upper bound for $\psi(n)/n$.

By the Euler product,
\[
\frac{\psi(n)}n
=
\prod_{\substack{p\le g(n)\\p\mid n}}
\left(1+\frac1p\right)
\prod_{\substack{p>g(n)\\p\mid n}}
\left(1+\frac1p\right).
\]

For the large-prime part,
\[
\begin{aligned}
\log
\prod_{\substack{p>g(n)\\p\mid n}}
\left(1+\frac1p\right)
&=
\sum_{\substack{p>g(n)\\p\mid n}}
\log\left(1+\frac1p\right)\\
&\le
\sum_{\substack{p>g(n)\\p\mid n}}\frac1p\\
&=
h_1(n)
<
\delta(n).
\end{aligned}
\]
Therefore
\[
\prod_{\substack{p>g(n)\\p\mid n}}
\left(1+\frac1p\right)
<
e^{\delta(n)}
=
O(1).
\]

For the small-prime part,
\[
\prod_{\substack{p\le g(n)\\p\mid n}}
\left(1+\frac1p\right)
\le
\prod_{p\le g(n)}
\left(1+\frac1p\right).
\]
By the standard Mertens-type product estimate
\cite{ref9,ref14},
\[
\prod_{p\le y}
\left(1+\frac1p\right)
=
\frac{6e^\gamma}{\pi^2}\log y
\left(
1+O\left(\frac1{\log y}\right)
\right).
\]
Taking $y=g(n)$, we obtain
\[
\prod_{\substack{p\le g(n)\\p\mid n}}
\left(1+\frac1p\right)
\ll
\log g(n).
\]
Since
\[
g(n)
=
c\,\frac{\log_2n}{\log_3n},
\]
we have
\[
\log g(n)
=
\log_3n-\log_4n+O(1)
\asymp
\log_3n.
\]

Thus
\begin{equation}
\frac{\psi(n)}n
\ll
\log_3n.
\label{eq:psiT-psi-upper}
\end{equation}

We now return to
\[
\frac{\psi(T(n))}{T(n)}
=
\frac{\psi(n)}n e^{\theta_n}.
\]
It follows that
\[
\left|
\frac{\psi(T(n))}{T(n)}
-
\frac{\psi(n)}n
\right|
=
\frac{\psi(n)}n
\left|e^{\theta_n}-1\right|.
\]
Since
\[
|\theta_n|
<
2\frac{\log_3n}{\log_2n}
\longrightarrow0,
\]
we may assume that $|\theta_n|\le1$ for all sufficiently large $n$.
Using
\[
|e^u-1|\ll|u|
\qquad(|u|\le1),
\]
we obtain
\[
\left|e^{\theta_n}-1\right|
\ll
\frac{\log_3n}{\log_2n}.
\]
Consequently,
\[
\left|
\frac{\psi(T(n))}{T(n)}
-
\frac{\psi(n)}n
\right|
\ll
\frac{(\log_3n)^2}{\log_2n}.
\]

Finally,
\[
\frac{T(T(n))}{T(n)}
=
\frac{\psi(T(n))}{T(n)}-1
\]
and
\[
\frac{T(n)}n
=
\frac{\psi(n)}n-1.
\]
Therefore
\[
\left|
\frac{T(T(n))}{T(n)}
-
\frac{T(n)}n
\right|
=
\left|
\frac{\psi(T(n))}{T(n)}
-
\frac{\psi(n)}n
\right|
\ll
\frac{(\log_3n)^2}{\log_2n}
=
o(1).
\]
Hence
\[
\frac{T^{(2)}(n)}{T(n)}
=
\frac{T(n)}n+o(1).
\]
\end{proof}

\section{Proof of Theorem~\ref{thm:lower}}
\label{sec:lower-bound}

\begin{proof}
Put
\[
\rho(n)
:=
\frac{T(n)}n
=
R(n)-1.
\]

We first record a truncation estimate. For $y\ge2$,
\[
\frac{R(n)}{R_y(n)}
=
\prod_{\substack{p>y\\p\mid n}}
\left(1+\frac1p\right),
\]
and hence
\[
0
\le
\log\frac{R(n)}{R_y(n)}
\le
\sum_{\substack{p>y\\p\mid n}}\frac1p
=
H_y(n).
\]
Therefore, if $H_y(n)\le\theta$, then
\[
R_y(n)\ge R(n)e^{-\theta}.
\]

Let now $0<\varepsilon<1$ and $\eta>0$ be fixed. We first choose
$Y$ so as to keep $\rho(n)$ away from $0$. By the Mertens product
formula \cite{ref9},
\[
\prod_{p\le Y}\left(1-\frac1p\right)\longrightarrow0
\qquad(Y\to\infty),
\]
so $Y$ may be chosen sufficiently large that
\[
\prod_{p\le Y}\left(1-\frac1p\right)<\eta.
\]
Put
\[
M_Y
:=
\prod_{p\le Y}p.
\]
The integers having no prime divisor $p\le Y$ are precisely those
coprime to $M_Y$. Since $M_Y$ is fixed,
\[
\#\{n\le x:(n,M_Y)=1\}
=
\frac{\varphi(M_Y)}{M_Y}x+O(M_Y).
\]
Thus this set has asymptotic density
\[
\frac{\varphi(M_Y)}{M_Y}
=
\prod_{p\le Y}\left(1-\frac1p\right)
<
\eta.
\]

Consequently, outside a set of asymptotic density less than $\eta$,
$n$ has a prime divisor $p\le Y$. For such $n$,
\[
\rho(n)
=
R(n)-1
\ge
\frac1p
\ge
\frac1Y,
\]
and hence
\[
\frac{R(n)}{\rho(n)}
=
1+\frac1{\rho(n)}
\le
1+Y.
\]

We next choose $\theta>0$ sufficiently small that
\[
\theta(1+Y)<\varepsilon.
\]
The role of the cutoff $y$ is different: it is used to make the
large-prime tail $H_y(n)$ small. By
Lemma~\ref{lem:large-prime-tail},
\[
\limsup_{x\to\infty}
\frac1x
\#\{n\le x:H_y(n)>\theta\}
\ll
\frac1{\theta y}.
\]
Since $\theta$ is fixed, we may choose a fixed $y\ge Y$ sufficiently
large that
\[
\limsup_{x\to\infty}
\frac1x
\#\{n\le x:H_y(n)>\theta\}
<
\eta.
\]

Since $y$ is now fixed,
Lemma~\ref{lem:small-prime-persistence}, applied with $K=k$, gives,
outside a set of asymptotic density zero,
\[
p\mid T^{(j)}(n)
\quad\Longleftrightarrow\quad
p\mid n
\]
simultaneously for every prime $p\le y$ and every
$0\le j\le k$. Hence
\[
R_y\bigl(T^{(j)}(n)\bigr)
=
R_y(n),
\qquad
0\le j\le k,
\]
simultaneously.

Moreover, by Lemma~\ref{lem:congruence-persistence},
\[
T^{(j)}(n)>1,
\qquad
0\le j\le k,
\]
outside a set of asymptotic density zero. Hence all the successive
ratios considered below are well defined on the resulting good set.

Let $n$ satisfy all the preceding good conditions. Then, for every
$0\le j<k$,
\[
\begin{aligned}
\frac{T^{(j+1)}(n)}{T^{(j)}(n)}
&=
R\bigl(T^{(j)}(n)\bigr)-1\\
&\ge
R_y\bigl(T^{(j)}(n)\bigr)-1\\
&=
R_y(n)-1\\
&\ge
R(n)e^{-\theta}-1\\
&=
\rho(n)-R(n)(1-e^{-\theta}).
\end{aligned}
\]
Since
\[
1-e^{-\theta}\le\theta,
\]
we obtain
\[
\frac{T^{(j+1)}(n)}{T^{(j)}(n)}
\ge
\rho(n)-\theta R(n).
\]
Using
\[
\frac{R(n)}{\rho(n)}
\le
1+Y
\]
and
\[
\theta(1+Y)<\varepsilon,
\]
we get
\[
\begin{aligned}
\frac{T^{(j+1)}(n)}{T^{(j)}(n)}
&\ge
\rho(n)
\left(
1-\theta\frac{R(n)}{\rho(n)}
\right)\\
&\ge
\rho(n)\bigl(1-\theta(1+Y)\bigr)\\
&\ge
(1-\varepsilon)\rho(n)\\
&=
(1-\varepsilon)\frac{T(n)}n.
\end{aligned}
\]
This holds simultaneously for $0\le j<k$.

It remains to estimate the exceptional set. Let
\[
\mathcal E_{\mathrm{small}}
:=
\{n:p\nmid n\text{ for every prime }p\le Y\},
\]
and
\[
\mathcal E_{\mathrm{tail}}
:=
\{n:H_y(n)>\theta\}.
\]
Let $\mathcal E_{\mathrm{iter}}$ denote the union of the exceptional
sets arising from Lemmas~\ref{lem:small-prime-persistence} and
\ref{lem:congruence-persistence}. Then
\[
d(\mathcal E_{\mathrm{small}})<\eta,
\]
\[
\overline d(\mathcal E_{\mathrm{tail}})<\eta,
\]
and
\[
d(\mathcal E_{\mathrm{iter}})=0.
\]
Hence
\[
\overline d
\left(
\mathcal E_{\mathrm{small}}
\cup
\mathcal E_{\mathrm{tail}}
\cup
\mathcal E_{\mathrm{iter}}
\right)
\le
2\eta.
\]

For fixed $\varepsilon>0$, let $\mathcal B_\varepsilon$ be the set
of positive integers $n$ for which either one of the relevant
successive ratios
\[
\frac{T^{(j+1)}(n)}{T^{(j)}(n)},
\qquad
0\le j<k,
\]
is not defined, or
\[
\frac{T^{(j+1)}(n)}{T^{(j)}(n)}
<
(1-\varepsilon)\frac{T(n)}n
\]
for at least one $0\le j<k$.

The preceding argument gives
\[
\mathcal B_\varepsilon
\subseteq
\mathcal E_{\mathrm{small}}
\cup
\mathcal E_{\mathrm{tail}}
\cup
\mathcal E_{\mathrm{iter}},
\]
and therefore
\[
\overline d(\mathcal B_\varepsilon)
\le
2\eta.
\]
Although the auxiliary sets on the right depend on $\eta$, the set
$\mathcal B_\varepsilon$ does not. Since $\eta>0$ is arbitrary,
\[
\overline d(\mathcal B_\varepsilon)=0.
\]
Moreover,
\[
0\leq \underline d(\mathcal B_\varepsilon)
\leq \overline d(\mathcal B_\varepsilon)=0.
\]
Hence the asymptotic density of $\mathcal B_\varepsilon$ exists and
\[
d(\mathcal B_\varepsilon)=0.
\]

We now let
\[
\varepsilon=\frac1m,\qquad m=2,3,\ldots.
\]
For $m\ge2$, put
\[
\mathcal C_m:=\mathcal B_{1/m}.
\]
Then
\[
\mathcal C_m\subseteq\mathcal C_{m+1},
\]
and, by the preceding argument,
\[
d(\mathcal C_m)=0
\qquad(m\ge2).
\]
Hence, for each $m\ge2$, we may choose an integer $N_m$,
with
\[
N_2<N_3<\cdots,\qquad N_m\to\infty,
\]
such that
\[
\#\bigl(\mathcal C_m\cap[1,x]\bigr)
\le \frac{x}{m}
\qquad(x\ge N_m).
\]

For $n\ge N_2$, let $m(n)$ be the unique integer satisfying
\[
N_{m(n)}\le n<N_{m(n)+1},
\]
and define
\[
\mathcal G
:=
\{n\ge N_2:n\notin\mathcal C_{m(n)}\}.
\]
We claim that $\mathcal G$ has asymptotic density one.
Indeed, if
\[
N_M\le x<N_{M+1},
\]
then, since
\[
\mathcal C_2\subseteq\mathcal C_3\subseteq\cdots\subseteq\mathcal C_M,
\]
we have
\[
\#\bigl([1,x]\setminus\mathcal G\bigr)
\le
N_2+\#\bigl(\mathcal C_M\cap[1,x]\bigr)
\le
N_2+\frac{x}{M}.
\]
Since $M\to\infty$ as $x\to\infty$, it follows that
\[
d(\mathcal G)=1.
\]

Now let $n\in\mathcal G$. Since
\[
n\notin\mathcal C_{m(n)}
=\mathcal B_{1/m(n)},
\]
we have, simultaneously for every $0\le j<k$,
\[
\frac{T^{(j+1)}(n)}{T^{(j)}(n)}
\ge
\left(1-\frac1{m(n)}\right)\frac{T(n)}n.
\]
Since $m(n)\to\infty$ as $n\to\infty$, the function
\[
\xi(n):=\frac1{m(n)}
\]
satisfies $\xi(n)\to0$ on $\mathcal G$. Therefore
\[
\frac{T^{(j+1)}(n)}{T^{(j)}(n)}
\ge
(1-\xi(n))\frac{T(n)}n,
\qquad 0\le j<k,
\]
simultaneously for every $n\in\mathcal G$.

Therefore
\[
\frac{T^{(j+1)}(n)}{T^{(j)}(n)}
\ge
(1-o(1))\frac{T(n)}n,
\qquad
0\le j<k,
\]
simultaneously for almost all $n$. Since $k$ is fixed, the $o(1)$
term may be chosen uniformly in $0\le j<k$.

Finally, let $j\ge1$ be fixed. Applying the preceding result with
$k=j$, there exists a function $\xi_j(n)\to0$ such that, for almost
all $n$,
\[
\frac{T^{(r+1)}(n)}{T^{(r)}(n)}
\ge
(1-\xi_j(n))\frac{T(n)}n,
\qquad
0\le r<j.
\]
Multiplying these inequalities gives
\[
\begin{aligned}
\frac{T^{(j)}(n)}n
&=
\prod_{r=0}^{j-1}
\frac{T^{(r+1)}(n)}{T^{(r)}(n)}\\
&\ge
(1-\xi_j(n))^j
\left(\frac{T(n)}n\right)^j\\
&=
(1-o(1))
\left(\frac{T(n)}n\right)^j,
\end{aligned}
\]
since $j$ is fixed. Therefore
\[
T^{(j)}(n)
\ge
(1-o(1))\,n
\left(\frac{T(n)}n\right)^j.
\]
\end{proof}

\section{Proof of Corollary~\ref{cor:periodic}}

\begin{proof}
Let
\[
\mathcal P_\ell
:=
\{n\ge1:T^{(\ell)}(n)\text{ is defined and }T^{(\ell)}(n)=n\}.
\]
Thus $\mathcal P_\ell$ is the set of periodic points whose exact
period divides $\ell$.

To count the elements of $\mathcal P_\ell$, we first choose one
distinguished element from each periodic orbit. Define
\[
\mathcal M_\ell
:=
\left\{
m\in\mathcal P_\ell:
m=\min_{0\le j<\ell}T^{(j)}(m)
\right\}.
\]
Thus $\mathcal M_\ell$ consists of the smallest integers in the
periodic orbits represented in $\mathcal P_\ell$.

We first show that $\mathcal M_\ell$ has asymptotic density zero.
Fix
\[
\varepsilon\in(0,1).
\]
Let $E_{\varepsilon,\ell}$ denote the set of positive integers for
which one of the relevant successive ratios is undefined, or for
which
\[
\frac{T^{(j+1)}(n)}{T^{(j)}(n)}
<
(1-\varepsilon)\frac{T(n)}n
\]
for at least one
\[
0\le j<\ell.
\]
By Theorem~\ref{thm:lower}, applied with $k=\ell$,
\[
d(E_{\varepsilon,\ell})=0.
\]
Consequently, for every $n\notin E_{\varepsilon,\ell}$,
\[
\frac{T^{(j+1)}(n)}{T^{(j)}(n)}
\ge
(1-\varepsilon)\frac{T(n)}n,
\qquad
0\le j<\ell.
\]

Now let
\[
m\in
\mathcal M_\ell\setminus E_{\varepsilon,\ell}.
\]
By the definition of $\mathcal M_\ell$,
\[
T(m)\ge m
\]
and
\[
T^{(\ell-1)}(m)\ge m.
\]
Moreover, since $m\in\mathcal P_\ell$,
\[
T^{(\ell)}(m)=m.
\]
It follows that
\[
\frac{T(m)}m\ge1
\]
and
\[
\frac{T^{(\ell)}(m)}
{T^{(\ell-1)}(m)}
=
\frac{m}{T^{(\ell-1)}(m)}
\le1.
\]

Since $m\notin E_{\varepsilon,\ell}$, taking $j=\ell-1$ gives
\[
\frac{T^{(\ell)}(m)}
{T^{(\ell-1)}(m)}
\ge
(1-\varepsilon)\frac{T(m)}m.
\]
Therefore
\[
1
\ge
(1-\varepsilon)\frac{T(m)}m,
\]
and hence
\[
1
\le
\frac{T(m)}m
\le
\frac1{1-\varepsilon}.
\]
Since
\[
\frac{T(m)}m
=
\frac{\psi(m)}m-1,
\]
we obtain
\[
2
\le
\frac{\psi(m)}m
\le
2+\frac{\varepsilon}{1-\varepsilon}.
\]

Thus
\[
\mathcal M_\ell\setminus E_{\varepsilon,\ell}
\subseteq
\left\{
n:
2\le\frac{\psi(n)}n
\le
2+\frac{\varepsilon}{1-\varepsilon}
\right\}.
\]
Equivalently,
\[
\mathcal M_\ell
\subseteq
E_{\varepsilon,\ell}
\cup
\left\{
n:
2\le\frac{\psi(n)}n
\le
2+\frac{\varepsilon}{1-\varepsilon}
\right\}.
\]

By Lemma~\ref{lem:continuous-distribution}, $\psi(n)/n$ has a
continuous limiting distribution. Let $G$ denote its limiting
distribution function. Then, for every $\eta>0$,
\[
\left\{
n:
2\le\frac{\psi(n)}n\le2+\eta
\right\}
\subseteq
\left\{
n:
2-\eta<\frac{\psi(n)}n\le2+\eta
\right\}.
\]
Hence
\[
\limsup_{x\to\infty}
\frac1x
\#\left\{
n\le x:
2\le\frac{\psi(n)}n\le2+\eta
\right\}
\le
G(2+\eta)-G(2-\eta).
\]
Since $G$ is continuous at $2$,
\[
G(2+\eta)-G(2-\eta)
\longrightarrow0
\qquad(\eta\downarrow0).
\]

Using the preceding inclusion and
\[
d(E_{\varepsilon,\ell})=0,
\]
we obtain
\[
\begin{aligned}
\limsup_{x\to\infty}
\frac1x
\#\{m\le x:m\in\mathcal M_\ell\}
&\le
\limsup_{x\to\infty}
\frac1x
\#\left\{
n\le x:
2\le\frac{\psi(n)}n
\le
2+\frac{\varepsilon}{1-\varepsilon}
\right\}.
\end{aligned}
\]
Since
\[
\frac{\varepsilon}{1-\varepsilon}
\longrightarrow0
\qquad(\varepsilon\downarrow0),
\]
the continuity of $G$ at $2$ gives
\[
\limsup_{x\to\infty}
\frac1x
\#\{m\le x:m\in\mathcal M_\ell\}
=0.
\]
Therefore
\[
\#\{m\le x:m\in\mathcal M_\ell\}
=o(x).
\]

It remains to pass from the smallest elements of the periodic orbits
to all points of $\mathcal P_\ell$. If
$n\in\mathcal P_\ell$ and $r$ is its exact period, then
\[
r\mid\ell.
\]
Hence the periodic orbit of $n$ contains exactly
\[
r\le\ell
\]
distinct integers.

Suppose now that $n\le x$ belongs to such an orbit, and let $m$ be
the smallest integer in that orbit. Then
\[
m\le n\le x,
\]
so
\[
m\in\mathcal M_\ell,
\qquad
m\le x.
\]
Thus every periodic orbit contributing a point to
$\mathcal P_\ell\cap[1,x]$ has a representative in
$\mathcal M_\ell\cap[1,x]$, and each such orbit contains at most
$\ell$ distinct points. Therefore
\[
\#\{n\le x:T^{(\ell)}(n)=n\}
\le
\ell\,
\#\{m\le x:m\in\mathcal M_\ell\}
=
o(x).
\]

Thus $\mathcal P_\ell$ has asymptotic density zero. Since the set of
periodic points of exact period $\ell$ is a subset of
$\mathcal P_\ell$, it also has asymptotic density zero.
\end{proof}

\section{Higher Iterates}
\label{sec:higher}

Theorem~\ref{thm:second} suggests that the asymptotic stability of
the relative growth factor may persist through any fixed number of
iterations. This leads naturally to the following conjecture.

\begin{conjecture}\label{conj:higher}
Let $K$ be a fixed positive integer. Then, for almost all positive
integers $n$, simultaneously for all $0\le j<K$,
\[
\frac{T^{(j+1)}(n)}{T^{(j)}(n)}
=
\frac{T(n)}{n}+o(1).
\]
\end{conjecture}

\end{document}